\documentclass{article}
\usepackage{amsmath,hyperref}
\usepackage{pdflscape}
\usepackage{verbatim}
\usepackage{amsfonts}

\usepackage[ruled,vlined,linesnumbered]{algorithm2e}

\SetAlgorithmName{Algorithm Schema}{\autoref}{List of Algorithms}

\SetKwFor{ParallelForEach}{for each}{in parallel do}{endfor}

\newcommand{\R}{\mathbf{R}}

\usepackage{epstopdf}
 
\usepackage{amsthm} 
\theoremstyle{plain}
\newtheorem{theorem}{Theorem}[section]

\newtheorem{lemma}[theorem]{Lemma}

\newtheorem{assumption}[theorem]{Assumption}

\theoremstyle{definition}

\theoremstyle{remark}

\usepackage[colorinlistoftodos,bordercolor=orange,backgroundcolor=orange!20,linecolor=orange,textsize=scriptsize]{todonotes}

\usepackage{dcolumn}

\newcolumntype{d}[1]{D{.}{\cdot}{#1} }

\newcommand{\inTR}[1]{}

\SetCommentSty{\small\sf}

\newcommand{\tr}{\textrm{tr}}

   \newcommand{\mfR}{\Re} 
  \newcommand{\mfI}{\Im}

\usepackage{algorithmic}

\usepackage{cite}

\DeclareGraphicsExtensions{.pdf,.jpeg,.png}
\usepackage{algorithmic}
\usepackage{array,color}
\usepackage[caption=false,font=footnotesize]{subfig}
\usepackage{dblfloatfix}
\usepackage{url}

\newcommand{\cL}{{\cal{L}}}
\newcommand{\cN}{{\cal N}}

\newcommand{\cG}{{\cal G}}
\newcommand{\cg}{g}

\newcommand{\cE}{{\cal E}}

\newcommand{\cM}{{\cal M}}

\newcommand{\cY}{{\cal Y}}

\newcommand{\bu}{{\bf u}}

\newcommand{\bv}{{\bf v}}
\newcommand{\bi}{{\bf i}}

\newcommand{\sfT}{\textsf{T}}

\begin{document}

\title{A Coordinate-Descent Algorithm \\ for Tracking Solutions \\ in Time-Varying Optimal Power Flows}

\author{Jie Liu, Jakub Mare\v{c}ek, Andrea Simonetto, Martin Tak\'a\v{c}}

\maketitle

\begin{abstract}
Consider a polynomial optimisation problem, whose instances vary continuously over time. 
We propose to use a coordinate-descent algorithm for solving 
such time-varying optimisation problems. 
In particular, we focus on 
relaxations of transmission-constrained problems in power systems.

On the example of the alternating-current optimal power flows (ACOPF),
we bound the difference between the current approximate optimal cost generated by our algorithm 
and the optimal cost for a relaxation using the most recent data from above 
by a function of the properties of the instance and the rate of change to the instance over time. 
We also bound the number of floating-point operations that need to be performed between two 
updates in order to guarantee the error is bounded from above by a given constant.
\end{abstract}


\section{Introduction}

Renewable energy sources (RESs) have posed a number of novel challenges within the analysis and control of power systems.
Notably, when RESs are widely deployed and inject all available power, power quality and reliability may suffer.
In distribution systems, overvoltages may become more common. 
In both distribution and transmission systems, fast variations of power output may 
introduce power-flow reversals, unexpected losses, and transients,
which current systems are not tested to cope with. 
One hence needs to design real-time control mechanisms,
 e.g., to curtail real power at inverters of RESs, 
 while considering transmission constraints.
 
The complication is that transmission-constrained problems in the alternating-current model 
 are non-linear and non-convex.
In model-predictive control, approaches applying Newton method to the non-convex problem in a rolling-horizon fashion   
often perform well in practice, as long as the changes are limited, but provide little or no theoretical guarantees as to their performance,
 more generally. 
In contrast, solutions to certain relaxations \cite{lavaei2012zero,7024950} 
coincide with the solutions to the non-convex problems, under mild assumptions, for all initial points.
The complication there is that solving the relaxation may take long enough for the 
 inputs to change considerably, making the solution out-dated, when available.

This calls for the provision of time-varying solutions of time-varying optimisation problems. 
In this paper, we propose to use a coordinate-descent algorithm \cite{Marecek2017,Liu2017}, where each step has a closed-form solution,
in solving time-varying optimisation problems.
In the case of alternating-current optimal power flows (ACOPF),
we are able to bound the difference between the current approximate cost $\cL^k$ and the current 
optimum $\cL^{k,*}$ of the relaxation derived using the most recent update in expectation, i.e.,
$\lim\sup_{k \to \infty} \mathbb{E}[\cL^k - \cL^{k,*}]$,
from above as a function of the properties of the instance
and a bound on the extent of updates to the instance.
 

This provides a novel perspective on time-varying optimisation in power systems in two ways:
First, we do not consider a linearization \cite{7480375}, but rather the non-convex non-linear problem. %
In our analysis, we assume a variant of the Polyak-{\L}ojasiewicz condition, rather than (strong) convexity.
Second, the delay in applying the update is $O(np)$ for $n$ nodes connected to at most $p$ other nodes each, 
 thanks to the closed-form solution for each coordinate-wise step.
As we demonstrate in computational illustrations on the IEEE 37-node test feeder, 
tracking of ACOPF solutions is possible in practice.

\section{The Problem}


In keeping with recent literature \cite{lavaei2012zero,molzahn2011,7024950,Marecek2017}, and without any loss of generality, 
we consider the two-terminal pi-equivalent model of a power system 
with nodes $\cN := \{1,\ldots,N\}$ connected by lines $\cE := \{(m,n)\} \subset \cN  \times \cN$. 
A subset of nodes, $\cG \subseteq \cN$, are the controllable generators, $N_{\cG}:= |\cG|$.
We assume that time is discretized to $k \tau$, with multiplier $k \in \mathbb{N}$ and period $\tau > 0$ 
chosen to capture the variations on loads and ambient conditions. 
We consider the following variables:
\begin{itemize}
\item $V_n^k \in \mathbb{C}$ denotes the phasors for the line-to-ground voltage at the $k$th time period
\item $I_n^k \in \mathbb{C}$ denotes current injected at node $n$ over the $k$th time period
\item $P_{n}^k$ and $Q_n^k$ denote the active and reactive powers injected at $n \in \cG$ over the $k$th time period
\end{itemize}
which can be concatenated into $N$-dimensional complex vectors  $V^k := [V_1^k, \ldots, V_N^k]^\sfT \in \mathbb{C}^{N}$ 
and $I^k := [I_1^k, \ldots, I_N^k]^\sfT \in \mathbb{C}^{N}$. 
By combining Ohm's and Kirchhoff's circuit laws, one can obtain the usual:
$
I^k
= 
y
V^k
$,
where $y \in \mathbb{C}^{(N) \times (N)}$ is the system admittance matrix.
For one node, we can fix the voltage magnitude $\rho_0$ and angle, $V_0^k = \rho_0 e^{\mathrm{j} \theta_0}$,
at any time $k$.
As usual, we assume load is constant at each time $k$, where 
$P_{\ell,n}^k$ and $Q_{\ell,n}^k$ denote the real and reactive demands at node $n \in \cN \setminus \cG$ at time $k$. 
At generator $n \in \cG$, we assume $P_{\textrm{av},n}^{k}$ denotes the maximum active power generation at time $k$.
For example, in a PV system, $P_{\textrm{av},n}^{k}$ is a function of the irradiance, bounded from above by a
limit on the inverter.

Traditionally, one considers an off-line optimisation problem, known as the alternating-current optimal power flow (AC OPF), 
which can be cast in its simplest form at time $k \tau$ as:
\begin{subequations} 
\label{Pmg}
\begin{align} 
 \mathrm{(OPF}^k \mathrm{)} &  \min_{\bv, \bi,  \{P_i, Q_i \}_{i \in \cG} } \,\, h^k(\{V_i\}_{i \in \cN}) + \sum_{i \in \cG} f_i^k(P_i, Q_i)  \label{mg-cost} \\ 
\mathrm{s.t.} I^k & = y V^k \label{eq:iYv} \\ 
          V_i I_i^* & = P_i - P_{\ell,i}^k + \mathrm{j} (Q_i - Q_{\ell,i}^k), \hspace{.55cm}  \forall \, i \in \cG   \label{mg-balance-I} \\
          V_n I_n^* & = - P_{\ell,n}^k - \mathrm{j} Q_{\ell,n}^k, \hspace{2cm}  \forall \, n \in \cN \backslash \cG \label{mg-balance-L} \\
   V^{\mathrm{min}} & \leq |V_i| \leq V^{\mathrm{max}} ,  \hspace{2.2cm}  \forall \, i \in \cM  \hspace{0.1cm} \label{mg-Vlimits} \\
                  0 & \leq {P}_{n}  \leq  \min \{ P_{\textrm{av},n}^k, S_{n} \} \hspace{1.2cm}  \forall \, n \in \cN \\
           {Q}_{n}  & \leq  S_{n}, \forall \ n \in \cG  \label{mg-PV} 
 \end{align}
\end{subequations} 
where $S_n$  is the rated apparent power. 
where 
$V^{\mathrm{min}}$ and $V^{\mathrm{max}}$ are voltage limits, 
$\cM \subseteq \cN$ is a set of nodes where voltage regulation can be performed, 
$f_i^k(P_i, Q_i)$ is a time-varying function specifying performance objectives for the $i$th generator,
and $h^k(\{V_i\}_{i \in \cN})$ captures system-level objectives.

The simplest form of the ACOPF can be lifted in a higher dimension \cite{Marecek2017}.
For notational convenience, we skip the time index $k$, where not needed. 
Let us have a number of $2n\times 2n$ matrices,
\begin{align}
y_i &:= e_i e_i^\sfT y\\
Y_i &:= \frac12 
\begin{bmatrix}
\mfR(y_i + y_i^\sfT) & \mfI(y_i^\sfT - y_i)\\
\mfI(y_i - y_i^\sfT) & \mfR(y_i + y_i^\sfT) 
\end{bmatrix}
\label{defYk}
\\ 
\bar{Y}_i &:= -\frac12 
\begin{bmatrix}
\mfI(y_i + y_i^\sfT) & \mfR(y_i - y_i^\sfT)\\
\mfR(y_i^\sfT - y_i) & \mfI(y_i + y_i^\sfT) 
\end{bmatrix}
\label{defbarYk}
\\ 
M_i &:=  
\begin{bmatrix}
e_i e_i^\sfT & 0 \\
0 & e_i e_i^\sfT 
\end{bmatrix},
\label{defMk}
\end{align}
where $e_i$ is the $i^{th}$ standard basis vector. 
One can then introduce new variables:
\begin{align}
x   &:=  \begin{bmatrix}\mfR{V}\\ \mfI{V}\end{bmatrix} \\
t_i &:=  \tr(Y_i xx^\sfT), \forall i\in \mathcal{N} \label{deft} \\ 
g_i &:=  \tr(\bar Y_i xx^\sfT), \forall i\in \mathcal{N} \label{defg} \\ 
h_i &:=  \tr(M_i xx^\sfT), \forall i\in \mathcal{N}. \label{defh}
\end{align} %
Using variables $t_i, g_i, z_i, \forall i\in \mathcal{G}$ and $h_i, i\in\mathcal{N}, x \in \R^{2n}$,
we can reformulate the problem as:
\begin{subequations}
\label{lifted}
\begin{align}
\min_{x\in\mathbb{R}^{2|\mathcal{N}|}} &\sum_{i\in\mathcal{G}} c_i [P_{l,i} +\tr(Y_i xx^\sfT)]^2 + d_i [Q_{l,i} 
+\tr(\bar Y_i xx^\sfT)]^2 \label{lifted1} \\
\text{s.t. }  
t_i &= \tr(Y_i xx^\sfT), \forall i\in \mathcal{N} \label{lifted2} \\
g_i &= \tr(\bar Y_i xx^\sfT), \forall i\in \mathcal{N} \label{lifted3} \\
h_i &= \tr(M_i xx^\sfT), \forall i\in \mathcal{N} \label{lifted4} \\
V_{min}^2&\leq h_i \leq V_{max}^2, \forall i\in\mathcal{N}\\
z_i &= (P_{l,i} + t_i)^2 + (Q_{l,i} + g_i)^2, \forall i\in\mathcal{G}\label{lifted6}\\
z_i &\leq S_i^2, \forall i\in\mathcal{G},\\
-P_{l,i} &\leq  t_i \leq P_{pv}- P_{l,i}, \forall i\in\mathcal{G},\\
t_i &= -P_{l,i}, \forall i\in\mathcal{N}\backslash\mathcal{G}, \label{lifted9} \\
g_i &=-Q_{l,i}, \forall i\in \mathcal{N}\backslash\mathcal{G}. \label{lifted10}
\end{align}
\end{subequations} 
One can extend the problem further \cite{molzahn2011,Marecek2017} to consider tap-changing and phase-shifting transformers
in per-line thermal limits, but that is outside of the scope of the present paper. 


Considering that \eqref{Pmg} is a \emph{nonconvex} optimisation problem,
a relaxation is usually considered. 

\section{The Approach}

Our approach is based on first-order methods for the Lagrangian relaxation of \eqref{lifted}:
\begin{align*}
& \xi := (x, t, h, g, z, \lambda^t, \lambda^g, \lambda^h, \lambda^z), \\
&\cL( \xi, \mu ):=
\\&  \sum_{i\in\mathcal{G}} \left\{ c_i [P_{l,i} +\tr(Y_i xx^\sfT) ]^2 +   d_i [Q_{l,i}  +\tr(\bar Y_i xx^\sfT) ]^2\right\}
\\& - \sum_{i\in\mathcal{N}} \lambda_i^t \left[\tr(Y_i xx^\sfT)  - t_i\right] 
 + \frac\mu2\sum_{i\in\mathcal{N}} \left[\tr(Y_ixx^\sfT)-t_i\right]^2
\\& - \sum_{i\in\mathcal{N}} \lambda_i^g \left[\tr(\bar Y_i xx^\sfT)  - g_i \right] 
 + \frac\mu2\sum_{i\in\mathcal{N}} \left[\tr(\bar Y_i xx^\sfT) -g_i\right]^2
\\& - \sum_{i\in\mathcal{N}} \lambda_i^h \left[\tr(M_i xx^\sfT) - h_i \right] 
  + \frac\mu2\sum_{i\in\mathcal{N}} \left[\tr(M_i xx^\sfT)-h_i \right]^2
\\& - \sum_{i\in\mathcal{G}} \lambda_{i}^z \left[(t_i+P_{l,i})^2  + (g_i+Q_{l,i})^2 - z_{i}\right]   
\\& + \frac\mu2\sum_{i\in\mathcal{G}} \left[(t_i+P_{l,i})^2  + (g_i+Q_{l,i})^2 - z_{i}\right]^2.
\tag{AL} \label{AugLagrangian}
\end{align*} 
which is intimately connected to the semidefinite programming (SDP)
relaxations \cite{lavaei2012zero,7024950}, where $xx^\sfT$ is replaced by $W \succeq 0$,
as described in \cite{lavaei2012zero,Marecek2017}.
In particular, we optimize $\cL$ \eqref{AugLagrangian} over a polyhedral feasible set $\cY$ defined by:
\begin{align}
    \label{cY1}
-P_{l,i} &\leq  t_i \leq P_{pv}- P_{l,i}, \forall i\in\mathcal{G},\\
V_{min}^2&\leq h_i \leq V_{max}^2, \forall i\in\mathcal{N}\\
z_i &\leq S_i^2, \forall i\in\mathcal{G} \label{cY3}
\end{align} 
Let us denote $x, t, h, g, z, \lambda^t, \lambda^g, \lambda^h,$ $\lambda^z$
in iteration $k$ as $\xi^k$ in dimension $d$. The update of coordinate $i^k$ to obtain $\xi^{k+1}_{i^k}$ is
\begin{align}\label{eq:updateopt_compact}
\arg \min_{\alpha \in \R} \left [ \alpha \nabla_{i^k} \cL(\xi^k,\mu) + \frac{L }{ 2}\alpha^2 + \cg_{i^k}(\xi_{i^k} + \alpha) - \cg_{i^k}(\xi_{i^k}) \right ],
\end{align}
where $\nabla_{i^k} \cL$ is the gradient  restricted to coordinate $i^k$.
This could be seen as a coordinate-wise minimisation applied to:
\begin{align}
\label{f+g}
\arg \min_{\xi^k} \cL(\xi^k, \mu) + \cg(\xi^k),
\end{align}
where 
$\cg$ is an indicator function that is zero if $\xi_i$ lies in $\cY$ set and infinity otherwise.

Crucially, notice that there exists a closed-form solution for the step-size $\alpha$
in \eqref{eq:updateopt_compact}. 
Considering that $\cL$ is a quartic polynomial \eqref{AugLagrangian}  (in $x$),
the optimality conditions are cubic, the uni-variate problem
has a closed-form solution of each root.
These can be enumerated and the minimum chosen.
For other variables ($t,h,g,u,v,z$) the $\cL$
is at most quadratic with respect to simple constraints.
This allows for both excellent computational performance
and the analysis of the per-iteration complexity
in Section~\ref{sec:periteration}.

\section{Iteration Complexity}

Let us consider the properties of the time-varying gradient mapping $\nabla \cL$ in more detail first.

\begin{lemma}
\label{lemma-Phi}
Let $\Xi:=B_r(\xi^*) \subset \mathbb{R}^d$ be a 
Euclidean ball centered in $\xi^*$
with a radius $r < \infty$.
Then 
$\nabla_\xi \cL$ is coordinate-wise Lipschitz continuous on $\Xi$, i.e.,
there $\exists L < \infty$ such that 
$\forall \alpha \in \mathbb{R}$,
$\xi \in \Xi$ and $\forall i \in \{1,2,\dots,d\}$ 
such that $\xi + \alpha e_i \in \Xi$ 
the following upper-bound is satisfied
\begin{align}
\label{lip_coo}
 	\nabla \cL(\xi + \alpha e_i, \mu) \leq \nabla \cL(\xi, \mu) + \alpha \nabla_i \cL(\xi, \mu) + \frac{L }{ 2} \alpha^2,
\end{align}
where $e_i$ is the $i$-th unit vector. 
\end{lemma}
\begin{proof}
Indeed, for fixed and finite $\mu$, the function $\cL(\xi, \mu)$ is an analytical polynomial function
 (infinitely differentiable). One can then define
$$L := 
\max_{i,\xi \in \Xi}  \left| \frac{\partial^2 \cL(\xi,\mu)}{\partial \xi_i^2}  \right|.$$
This value will be finite because $i\in \{1,2,\dots,d\}$ is just a finite set,
$$\frac{\partial^2 \cL(\xi,\mu)}{\partial \xi_i^2} $$ is a polynomial function,
 and $\Xi$ is a compact set.
\end{proof}

To proceed with the analysis of the rates of convergence of gradient methods on \eqref{AugLagrangian}, 
one often makes a number of assumptions. 
Outside of convexity, one often assumes a variant of strong convexity,
such as essential strong convexity (ESC) of Liu and Wright \cite{liu2015asynchronous}
or weak strong convexity (WSC) \cite{necoara2015linear,ma2016linear}. 
Notice that strong convexity and its variants (ESC, WSC) imply the uniqueness of optima
and require that each stationary point is an optimum. 
Such an assumption may be hard to justify, considering that ACOPF is non-convex and its convex relaxations 
 need not have a unique optimum.
(Consider a case, where there are  two invertors with one and the same linear cost function, e.g., a feed-in tarrif,
 connected to a single load by one line each, with both lines having the same branch admittance.)
Instead, we make an assumption relating the growth of gradient to sub-optimality:

\begin{assumption}\label{asmPL}[\cite{Polyak63,lojasiewicz1963propriete,karimi2016linear}]
Given a local minimizer $\xi^*$, 
and a fixed $\mu \in [0,\bar \mu]$, there 
exists a positive $r < \infty$ and $\sigma_\cL > 0$ such that the map $\nabla \cL$ satisfies local proximal 
Polyak-\L{}ojasiewicz Inequality, i.e., $\forall \xi \in \Xi := B_r(\xi^*)$ the following inequality holds:
\begin{equation}\label{prox-pl}
\frac 1 2\mathcal{D}_\cg(\xi, L) \geq   \sigma_\cL (\cL(\xi,\mu) - \cL(\xi^*,\mu)),
\end{equation} 
where $g$ is the indicator function as above \eqref{f+g} and $\mathcal{D}_g(\xi, \alpha)$ is defined as follows:
\begin{equation}
-2\alpha \min_{\xi'} \left[ \langle \nabla \cL(\xi,\mu) , \xi' - \xi \rangle + \frac{\alpha}{2}|| \xi' - \xi ||^2 + \cg(\xi') - \cg(\xi) \right].
\nonumber
\end{equation}
\end{assumption}

Under this assumption, it is possible to show a linear rate of convergence of the randomized coordinate-descent algorithm 
considering the input at time $k$ as a constant.
Notice that due to the non-convex nature of the function $\cL$ in variable $\xi$,
 the analysis of the global convergence to a solution of the semidefinite programming (SDP) relaxation \cite{lavaei2012zero,7024950} 
 would have to exploit an additional regulariser.
Although this is well-known both in general \cite{Burer2003,burer2005} and within power systems analysis \cite{Marecek2017},
 it is somewhat technical, cf. Theorem 4.1 in \cite{burer2005} and its use in \cite{Marecek2017,boumal2016}.
In this paper, we hence limit ourselves to the simpler analysis of local convergence.
Let $\xi^*$ be any local minimizer of $\cL(\xi,\mu)$ for fixed $\mu$. 
We will assume throught this paper that $\mu \in [0, \bar \mu]$, with $\bar \mu < \infty$.
 
\begin{theorem}\label{th:lin_coo}[Extension of Theorem 6 in \cite{karimi2016linear}]
Let $\mu \in [0,\bar \mu]$ is fixed and
$\xi^*$, $r$ and $\sigma_\cL$ are such that 
Assumption~\ref{asmPL} is satisfied.
Moreover, let $\xi^0, \xi^1, \dots \in \Xi$.
Then the randomized coordinate-descent algorithm \eqref{eq:updateopt_compact}, with $i^k$ being chosen uniformly at random from $\{1,2,\dots,d\}$, 
for solving \eqref{f+g}
has a local linear convergence rate:
\begin{align}
\mathbb{E}[ \cL(\xi^k,\mu) - \cL^*] \leq \left( 1 - \frac{\sigma_\cL }{ d L}\right)^k[ \cL(\xi^0,\mu) - \cL^*],
\end{align}
where $L$ is as defined in Lemma~\ref{lemma-Phi} and $\cL^* : = \cL(\xi^*,\mu)$.
\end{theorem}

The proof follows from \cite{karimi2016linear}.

Now, let us bound the error in tracking, i.e., 
when $\cL$ changes over time due to time-varying input parameters and we run only one iteration of our algorithm per time step, 
before obtaining new inputs. 
Let us denote the time-varying $\cL$ at each time (sampling instance) $k$ by $\cL^k(\xi,\mu)$  
and make the following assumption: 

\begin{assumption}\label{as:varying}
The variation of the function $\cL^k$ at two subsequent instant $k$ and $k-1$ is upper bounded as 
$$
|\cL^{k}(\xi,\mu) - \cL^{k-1}(\xi,\mu)| \leq e, \quad\textrm{for all } \xi \in \cY
$$
for all instants $k>0$. 
\end{assumption}


Assumption~\ref{as:varying} bounds how the function $\cL$ changes over time and gives 
makes it possible to measure the tracking performance: 

\begin{theorem}
\label{theorem.inexact}
Let $\mu \in [0,\bar \mu]$ is fixed and
$\xi^{*,k}$, $r$ and $\sigma_\cL$ are such that 
Assumption~\ref{asmPL} as well as the Lipschitz condition~\eqref{lip_coo} are satisfied, uniformly in time. Let Assumption~\ref{as:varying} hold. 
Moreover, let $\xi^0, \xi^1, \dots \in \Xi$.
Then the randomized coordinate-descent algorithm \eqref{eq:updateopt_compact}, with $i^k$ being chosen uniformly at random from $\{1,2,\dots,d\}$, 
for solving \eqref{f+g} with $\cL^k(\xi,\mu)$ instead of $\cL(\xi,\mu)$ has a local linear convergence rate to an error bound as: 
\begin{multline}
\mathbb{E}[ \cL^{k}(\xi^k,\mu) - \cL^{*,k}] \\
\leq \left( 1 - \frac{\sigma_{\cL}}{ d L}\right)^k[ \cL^{0}(\xi^0,\mu) - \cL^{*,0}] + \frac{1}{1-\frac{\sigma_{\cL}}{dL}}\, e,
\end{multline}
while the tracking error is, 
\begin{align}\label{eq.asympt_error}
\limsup_{k \to \infty}\mathbb{E}[ \cL^{k}(\xi^k,\mu) - \cL^{*,k}] \leq \frac{1}{1-\sigma_{\cL}/dL}\, e.
\end{align}
\end{theorem}

\begin{proof}
The proof follows from Theorem~\ref{th:lin_coo}, by invoking the triangle inequality and 
the sum of a geometric series. In particular, dropping the dependency on $\mu$ for sake of compactness, one has for each $k$
\begin{align}
    \label{errorbound2use}
\mathbb{E}[\cL^{k-1}(\xi^k) - \cL^{*,k-1}] \leq \left(1 - \frac{\sigma_{\cL}}{dL}\right)[\cL^{k-1}(\xi^{k-1}) - \cL^{*,k-1}],
\end{align}
by summing and subtracting $\mathbb{E}[\cL^{k}(\xi^k)]$ on the left-hand-side and by 
putting w.l.g. $\cL^{*,k-1} = \cL^{*,k}$, 
\begin{multline}
\mathbb{E}[\cL^{k}(\xi^k) - \cL^{*,k}] \leq \left(1 - \frac{\sigma_{\cL}}{dL}\right)[\cL^{k-1}(\xi^{k-1}) - \cL^{*,k-1}] +\\ |\mathbb{E}[\cL^{k}(\xi^k) - \cL^{k-1}(\xi^{k})]|,
\end{multline}
which we can bound by Assumption~\ref{as:varying}. By the summation of geometric series, the claim is proven.
\end{proof}


Equation~\eqref{eq.asympt_error} quantifies the maximum discrepancy between the approximate optimum
 $\cL^k(\xi^k,\mu)$ and $\cL^{*,k}$ at instant $k$, as $k$ goes to infinity. 
In particular, as time passes, our on-line algorithm generates a sequence of approximately optimal costs 
that eventually reaches the optimal cost \emph{trajectory}, up to an asymptotic bound. 
The convergence to the bound is linear and depends on the properties of the cost function, 
while the asymptotic bound depends on how fast the problem is changing over time. 
This is a \emph{tracking} result: we are pursuing a time-varying optimum by a finite number 
of iterations, e.g., one, per time-step. 
If we could run a large number of iterations per each time step, then we would be back to a static case 
of Theorem~\ref{th:lin_coo} and we would not have a tracking error. 
This may not, however, be possible in settings, where inputs change faster than one can compute an
 iteration of the algorithm. 




\section{The Per-Iteration Complexity}
\label{sec:periteration}

Let us now consider the complexity of a single iteration of the coordinate-descent algorithm, 
or rather the complexity of one epoch of the coordinate-descent algorithm, i.e., the iterations 
going sequentially over each coordinate $i$ in $\xi$:

\begin{lemma}
\label{noFlops}
Coordinate descent going sequentially over each coordinate $i$ in $\xi$, 
performs $(32p+102) n^2 + (32p+116)n_g n  - 2n + (16p+92)n_g$ floating-point operations
plus $6(n+n_g)$ evaluations of roots of a univariate cubic polynomial.
The update of a single coordinate requires at most
$16(n_g+n)p+58n_g+51n-8$ floating-point operations and $6n$ evaluations of a root of a univariate cubic polynomial.
\end{lemma}

\begin{proof}
First, notice that the evaluation of the traces of high-dimensional quadratic forms
can exploit sparsity.
For instance, consider 
matrix $Y_i$ \eqref{defYk}, in whose definition $y_i = e_i e_i^\sfT y$ with the system admittance matrix $y$. 
The evaluation of the trace of the quadratic form $\tr(Y_i xx^\sfT)$ can be performed in at most 
$8p$ float-float operations where $p$ denotes the number of non-zero elements of 
the $k^{\textrm{th}}$ row of $y$, which is a constant, $p \ll n$, for all realistic power systems. 
Further, terms involving $M_i$ \eqref{defMk} can be simplified, e.g.,
$$\tr(M_i xx^\sfT) = x_i^2 + x_{i+|N|}^2,$$
so as to be evaluated in $3$ float-float operations, and in $1$ flop if $x_i$ or $x_{i+|N|}$ is a variable.

Next, recall that we are minimising coordinate-wise. 
Enumerating the local minima of $\min_x ax^4 + bx^3 + cx^2 +dx +e$ is the same as solving the cubic equation 
$4ax^3 + 3bx^2 + 2cx + d = 0$, which after $7$ float-float multiplications becomes 
$x^3 + (3b)/(4a)x^2 + (c)/(2a)x + d/(4a) = 0$.
Obviously, we have unconstrained optimization problems for $x$ and box-constrained quartic optimizaton problems for $t$ and $g$, 
both of which take similar cost to solve. 

Finally, let us sum the numbers up, term-wise.
Evaluating a single coordinate in term $\tr(Y_i xx^\sfT)$
has the same cost as evaluation for $\tr(Y_i xx^\sfT)$, which is $8p$ flops. 
It takes $11$ additional operations to compute the coefficients 
for $[\tr(Y_i xx^\sfT)+P_{l.i}]^2$, and $3$ operations 
when $Y_i$ replaced by $M_i$. Assuming that the number of generators is $n_g$, in total, we need 
$\{n_g[2(8p+11) + 3\times 5] + 5(n_g-1) + n[2(8p+11) + (1+3+1)] +5(n-1) + 3(n-1)+5\}$ + 
$\{n[ 3\times 3 + 2\times 2] + 3(n-1)\}$ +  
$\{7n_g +(n_g-1) + 7n_g +(n_g-1)+1\} + 4 
= 16(n_g+n)p+58n_g+51n-8$ flops for coefficient evaluations, where the first two brace-delimited summands come from
\eqref{lifted2}, \eqref{lifted3}, \eqref{lifted4} in quartic and quadratic terms, respectively, and the last one comes from \eqref{lifted6}. Considering each epoch performs
$2n$ such coordinate-wise iterations, it has a cost of $(32p+102)n^2 +(32p+116)n_gn-16n$ flops plus $6n$ evaluations of a root 
of a univariate cubic polynomial (root-evals).

Similarly, for $t_i, i\in\mathcal{G}$ \eqref{deft}, the evaluations of the coefficient only occur at the quadratic and quartic terms, 
where quadratic terms $\left[(t_i+P_{l,i})^2  + (g_i+Q_{l,i})^2 - z_{i}\right]$ and $\left[\tr(Y_iW)-t_i\right]^2$ take $6$ and 
$(8p+2)$ flops, respectively. 
The quartic term takes $11$ more operations. 
Per-epoch the update of $t_i$ comes at the cost 
of $n_g\{ [8p+2 +3] + [2] + [(6+7) + 5] + [2] + 5 + 3+3\} = (8p+38)n_g$ flops plus $3n_g$ root-evals. 
The same cost also applies to updates in $g$ \eqref{defg}.

Further, for $h_i, i\in\mathcal{N},z_i$ and $i\in\mathcal{G}$, we have box-constrained quadratic optimization problems, 
and it is not difficult to count that the evaluation of coefficients requires $12$ and $14$ flops, respectively, for per coordinate and solving 
a quadratic problem takes only $2$ flops. Thus per-epoch, the cost is $14n$ and $16n_g$ flops  
for $h_i, z_i$, respectively.
\end{proof}

In summary, the total cost for one epoch is $(32p+102) n^2 + (32p+116)n_g n  - 2n + (16p+92)n_g$ 
float-float operations (flops) plus $6(n+n_g)$ evaluations of a root of a cubic polynomial. 
Bounding the number of flops required to evaluate the root of a cubic polynomial is somewhat involved, as the computation
 requires taking the square and cubic roots of scalars. 
In a model of computation, where taking the root of a scalar requires 1 flop,
such as in the BSS machine \cite{blum1989theory},
the root of a cubic polynomial can be evaluated in 31 flops.
The update of a single coordinate in such a model hence requires at most
$16(n_g+n)p+58n_g+144n-8$ floating-point operations.



This makes it possible to bound the expected tracking error $\cL^{k}(\xi^k,\mu) - \cL^{*,k}$ by quantities, 
which are easier to reason about. In particular, let us consider the number of floating-point 
operations needed to perform between two updates of the inputs, in order to achieve a certain 
guarantee in terms of the error bound:

 \begin{theorem}
\label{V2}
Let Assumptions~\ref{asmPL} and \ref{as:varying} hold, with an upper bound $e$ on the magnitude of change between two successive inputs. 
Considering the number $p$ of other nodes any node can be connected to as a constant, 
and parametrising the result by the size of the level-set $\sigma_l := [ \cL^{0}(\xi^0,\mu) - \cL^{*,0}]$,
and a parameter $\sigma_p := \frac{1}{1-\frac{\sigma_{\cL}}{dL}}$,
the number of floating-point operations a BSS machine needs to be able to perform between two successive updates 
of the inputs to guarantee the error is bounded by $E := \mathbb{E}[ \cL^{k}(\xi^k,\mu) - \cL^{*,k}]$ is:
\begin{align}
    \label{flopserror}
(16(n_g+n)p+58n_g+144n-8) \frac{\log (E - \sigma_p e)}{\log \sigma_l}.
\end{align}
\end{theorem}

\begin{proof}
The linear convergence established in Theorems~\ref{th:lin_coo} and \ref{theorem.inexact} means 
$E$ is bounded by a function of $\sigma_p$ raised to the $k$th power.
In turn, $k$ is bounded from above by the ratio of the total number of flops between two updates and a
worst-case bound on the numbers of flops required per 1 coordinate-wise update,
which is $16(n_g+n)p+58n_g+144n-8$ by Lemma~\ref{noFlops}, i.e., $O(np)$.
By substituting $\sigma_p, \sigma_l$ into \eqref{errorbound2use}, solving for $\sigma_p^k$, substituting the ratio instead of $k$,
and taking the logarithm of both sides, we obtain the result.
\end{proof}

Considering that modern computers are not BSS machines, and their behaviour is rather complex, 
the bound \eqref{flopserror} may not be a perfect estimate of the actual run time,
but it does provide some guidance as to the requirements on computing resources.
Specifically, the run-time to a constant error bound grows with $O(np)$, 
when $\sigma_l$ and $\sigma_p$ are constant. 



\begin{figure}[t] 
\centering
\vspace{.25cm}
\includegraphics[width=.90\textwidth]{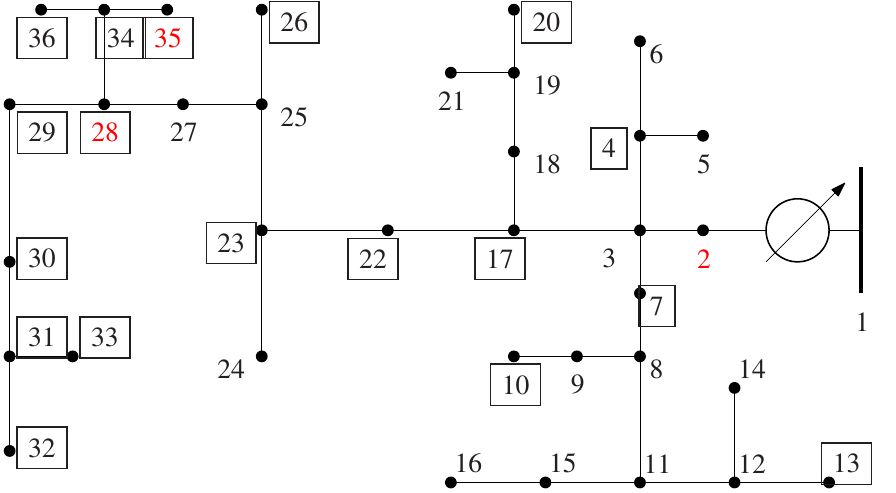}
\caption{IEEE 37-node feeder, as amended by Dall'Anese and Simonetto \cite{7480375}: 18 PV systems are marked with a box.}
\label{F_feeder}
\end{figure}

\section{Empirical Results}

We test our approach on a distribution network with high-penetration of photovoltaic (PV) systems, 
introduced by Dall'Anese and Simonetto \cite{7480375}, although our approach is by no means limited to radial networks.
The network is based on a single-phase variant of the IEEE 37-node test case. 
It replaces constant load of 18 secondary transformers
(at nodes $4$, $7$, $10$, $13$, $17$, $20$, $22$, $23$, $26$, $28$, $29$, $30$, $31$, $32$, $33$, $34$, $35$, and $36$, 
as highlighted in Figure~\ref{F_feeder}) with real load data from Anatolia, California, sampled with $1$ Hz frequency in August 2012~\cite{Bank13}. 
Further, the generation at PV plants is simulated based on real solar irradiance data in \cite{Bank13},
with rating of these inverters at $300$ kVA at node $3$; $350$ kVA at nodes $15, 16$, and $200$ kVA for all other inverters.
The voltage limits $V_{\mathrm{max}}$ and $V_{\mathrm{min}}$ are set to $1.05$ pu and $0.95$ pu, respectively.  
The solar irradiance data also have the granularity of $1$ second. 
Other parameters are kept intact.



Figures~\ref{fig:midnight} and \ref{fig:fivepm} present the performance
 evaluated at 3 Hz frequency, compared to the 1 Hz update.
On top, there is the voltage profile for nodes $2, 15, 28$, and $35$.
When compared to to Figure 4 by Dall'Anese and Simonetto \cite{7480375},
the voltage profiles seem much improved;   
there seems to be little volatility even in the zoomed-in Figure~\ref{fig:fivepm}.
In the middle plot, we present the achieved cost $\sum_{i \in \cG} c_q (Q_i^k)^2 + c_p (P_{\textrm{av},i}^k)^2$.
In the bottom plot, we present a measure of infeasibility:
\begin{align}
\label{infeasibility}
T(x, t, g, h, z) :=
 \sum_{i\in\mathcal{N}} \left[\tr(Y_ixx^\sfT)+\omega_i^\sfT x-t_i\right]^2 \nonumber\\
 + \sum_{i\in\mathcal{N}} \left[\tr(\bar Y_i xx^\sfT)+ \bar\omega_i^\sfT x-g_i\right]^2 
+\sum_{i\in\mathcal{N}} \left[\tr(M_i xx^\sfT)-h_i \right]^2 
\nonumber \\ 
+ \sum_{i\in\mathcal{G}} \left[(t_i+P_{l,i})^2  + (g_i+Q_{l,i})^2 - z_{i}\right]^2.
\nonumber
\end{align} 
and compare it against the linearisation of Dall'Anese and Simonetto~\cite{7480375}, 
wherein we use $\nu = 10^{-3}$, $\epsilon = 10^{-4}$, $\alpha = 0.2$,  $c_p = 3$,  $c_q = 1$, 
$\bar{f}^k(\bu^k) = \sum_{i \in \cG} c_q (Q_i^k)^2 + c_p (P_{\textrm{av},i}^k - P_i^k)^2$, 
as suggested by the authors.
For the linearisation, we evaluate both the full measure of infeasibility $T$ \eqref{infeasibility} and 
a lower bound $T'$ on the infeasibility \eqref{infeasibility}, which ignores 
the terms $\sum_{i\in\mathcal{N}\backslash\mathcal{G}} \left[\tr(Y_i xx^\sfT)+ \omega_i^\sfT x-t_i\right]^2 
+ \sum_{i\in\mathcal{N}\backslash\mathcal{G}} \left[\tr(\bar Y_i xx^\sfT)+ \bar\omega_i^\sfT x-g_i\right]^2$, 
which correspond to constraints \ref{lifted9} and \ref{lifted10} in the lifted formulation \eqref{lifted} and
to constraint \eqref{mg-balance-L} in the original formulation \eqref{Pmg} of \cite{7480375},
which is most affected by the linearisation.
Infeasibility $T$ of our approach is approximately 4 orders of magnitude better than 
the lower bound $T'$ on the infeasibility of the linearisation,
and about 8 orders of magnitude better than the infeasibility $T$ of the linearisation.

 \begin{figure}[tb]
 \center
\includegraphics[scale=0.5]{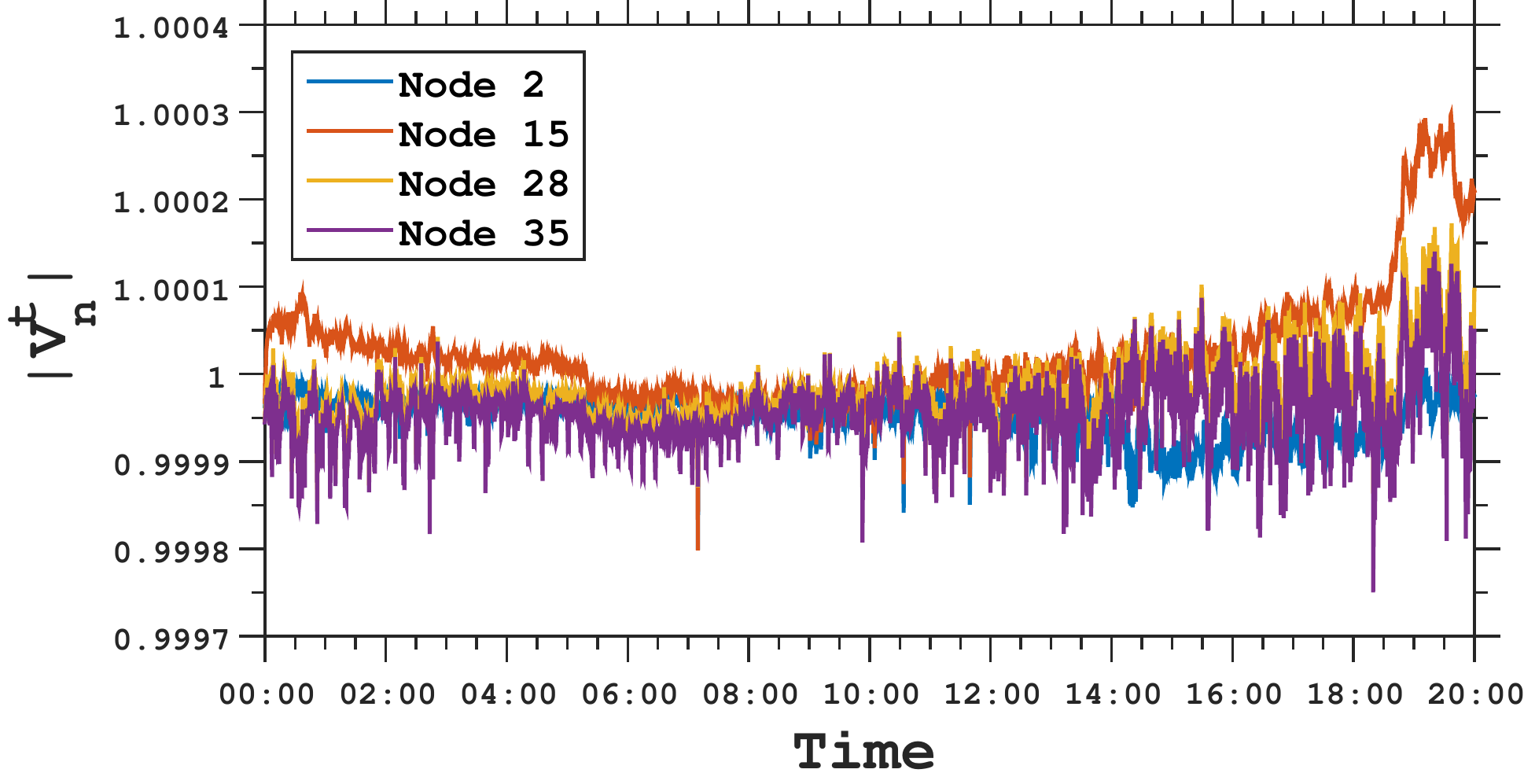}
\\ $ $\\
\vspace{-1em}
\qquad \includegraphics[scale=0.5]{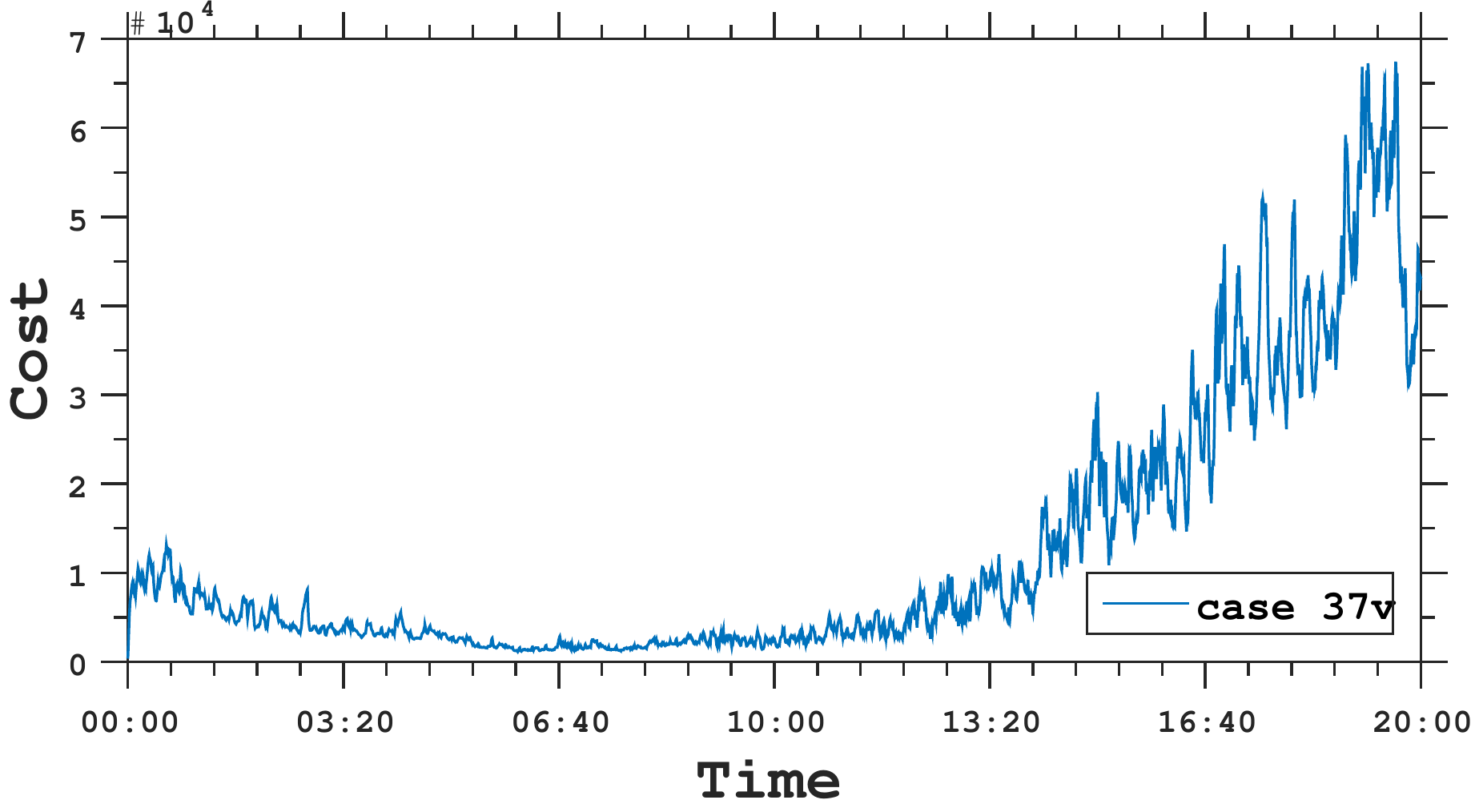}
\\ $ $\\
\vspace{-1em}
\quad \includegraphics[scale=0.5]{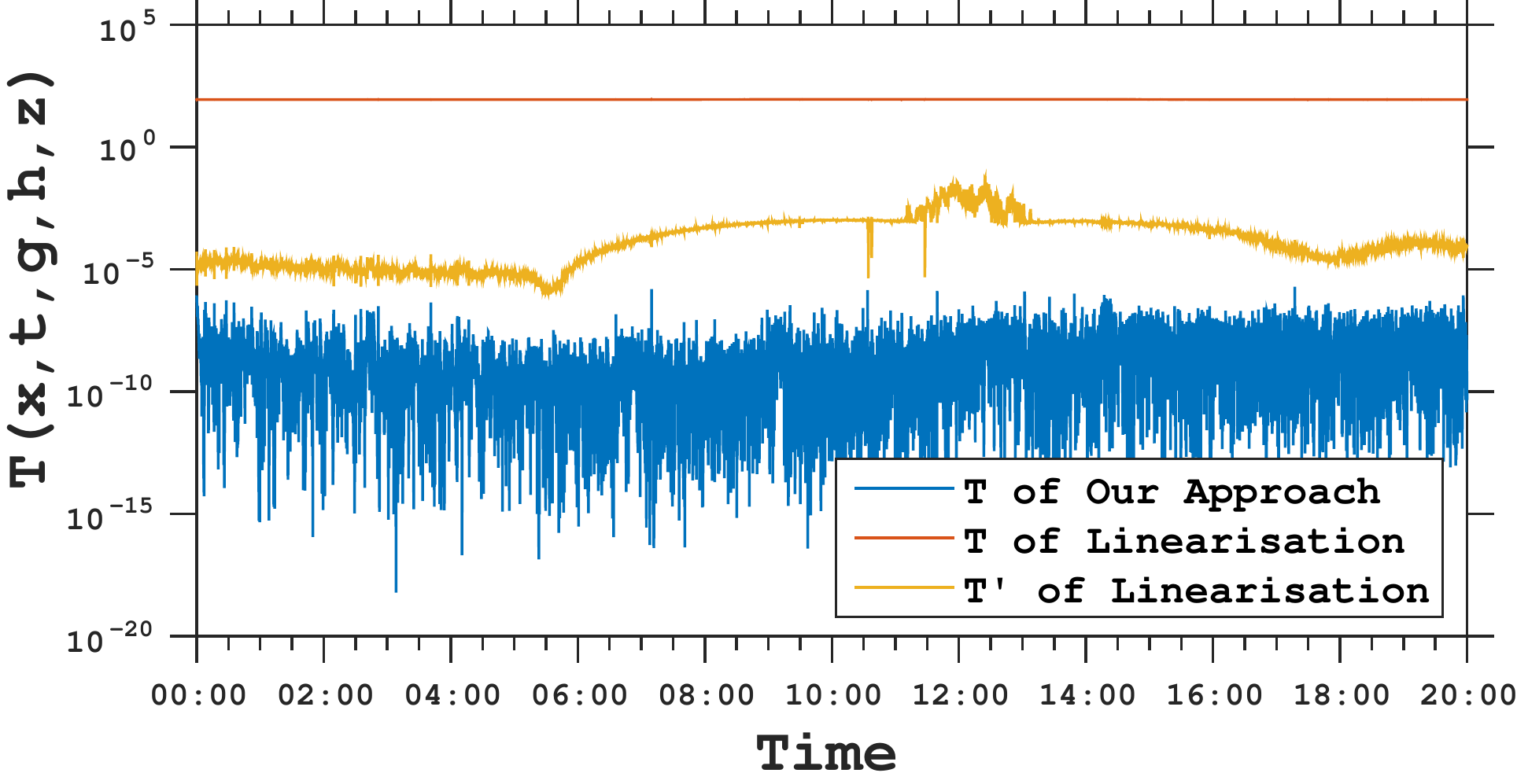}
\caption{The performance on the feeder of Figure~\ref{F_feeder}, from midnight till 8pm.}
 \label{fig:midnight}
\end{figure}

 \begin{figure}[tb]
 \center
\includegraphics[scale=0.5]{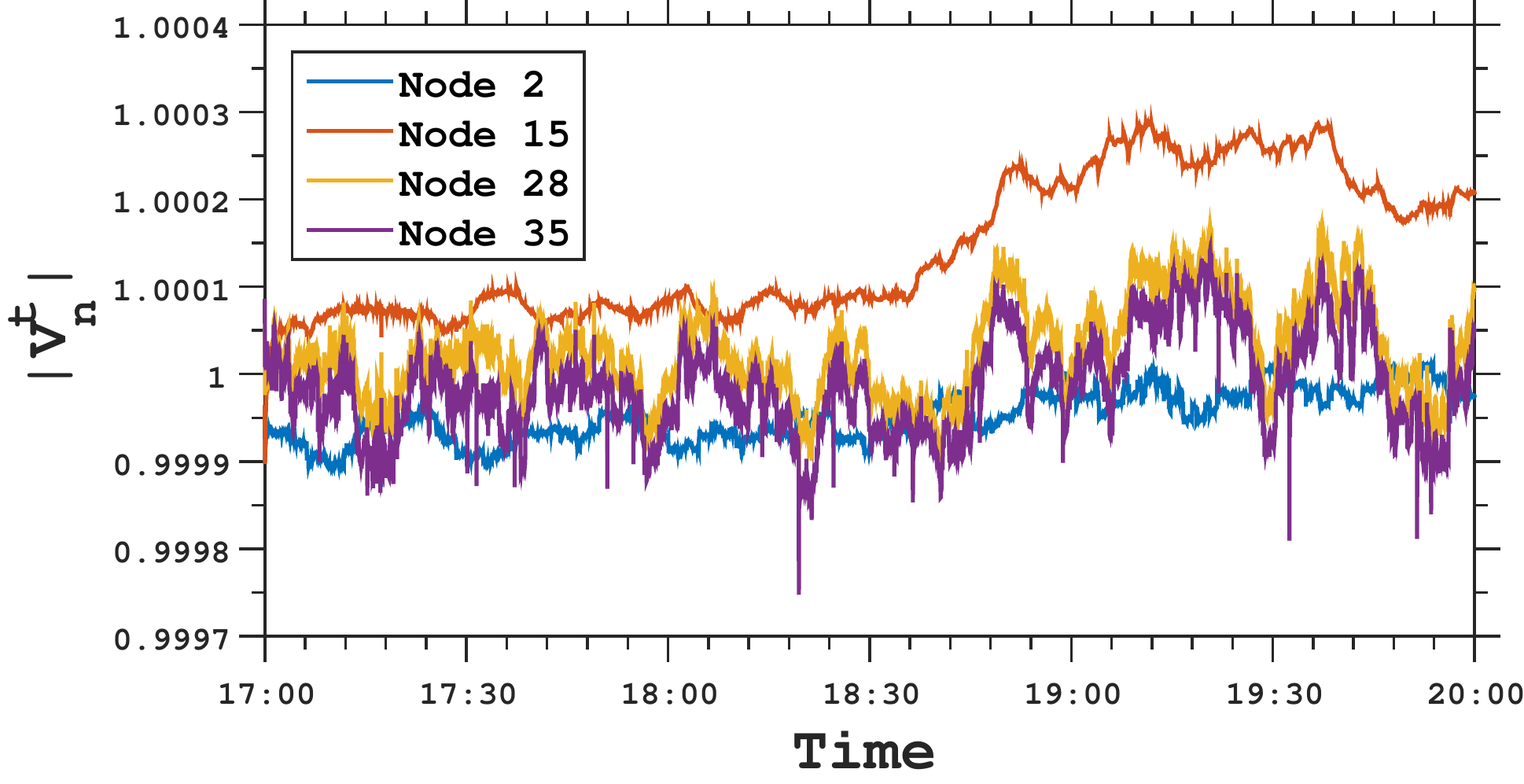}
\\ $ $\\
\vspace{-1em}
\qquad \includegraphics[scale=0.5]{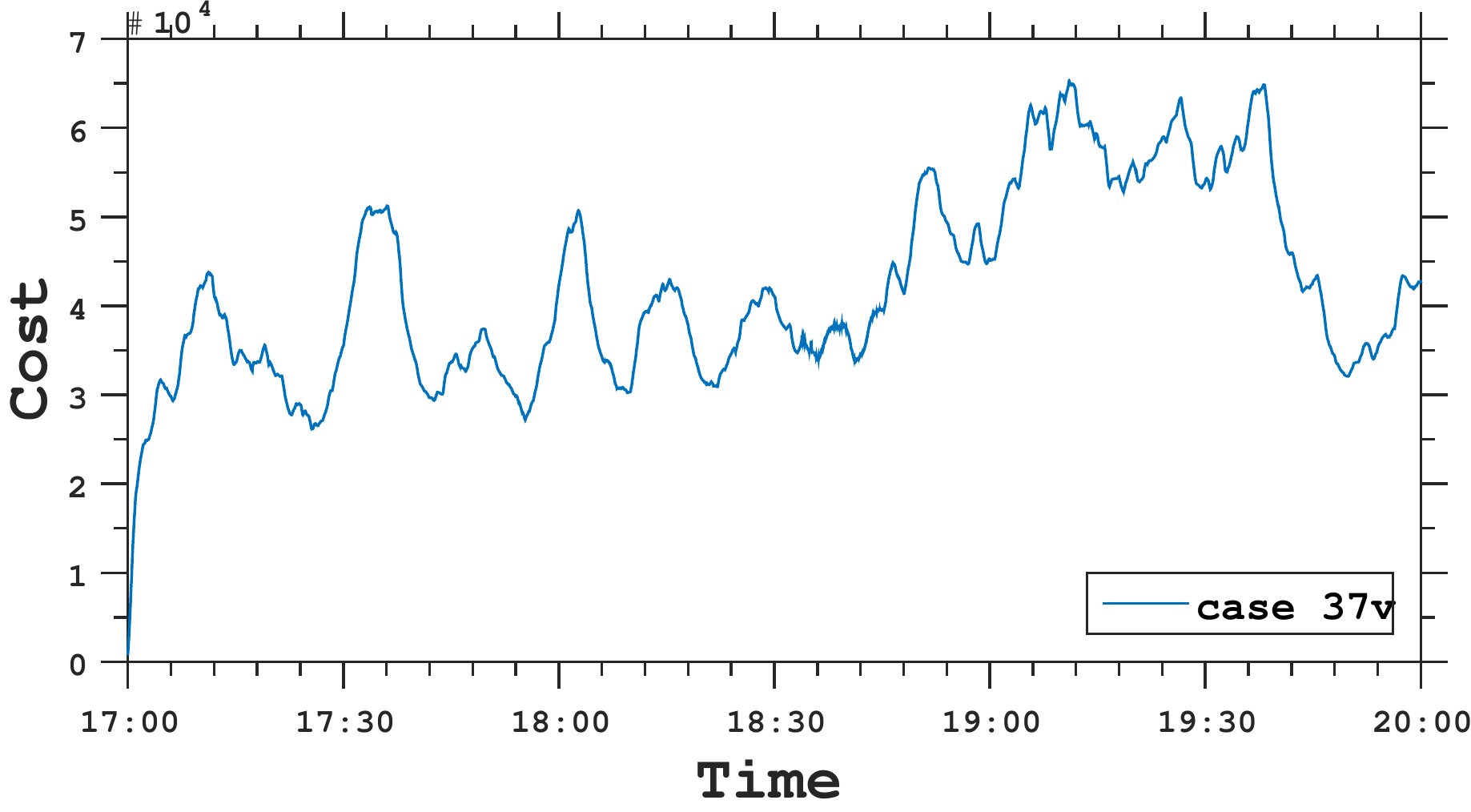}
\\ $ $\\
\vspace{-1em}
\quad \includegraphics[scale=0.5]{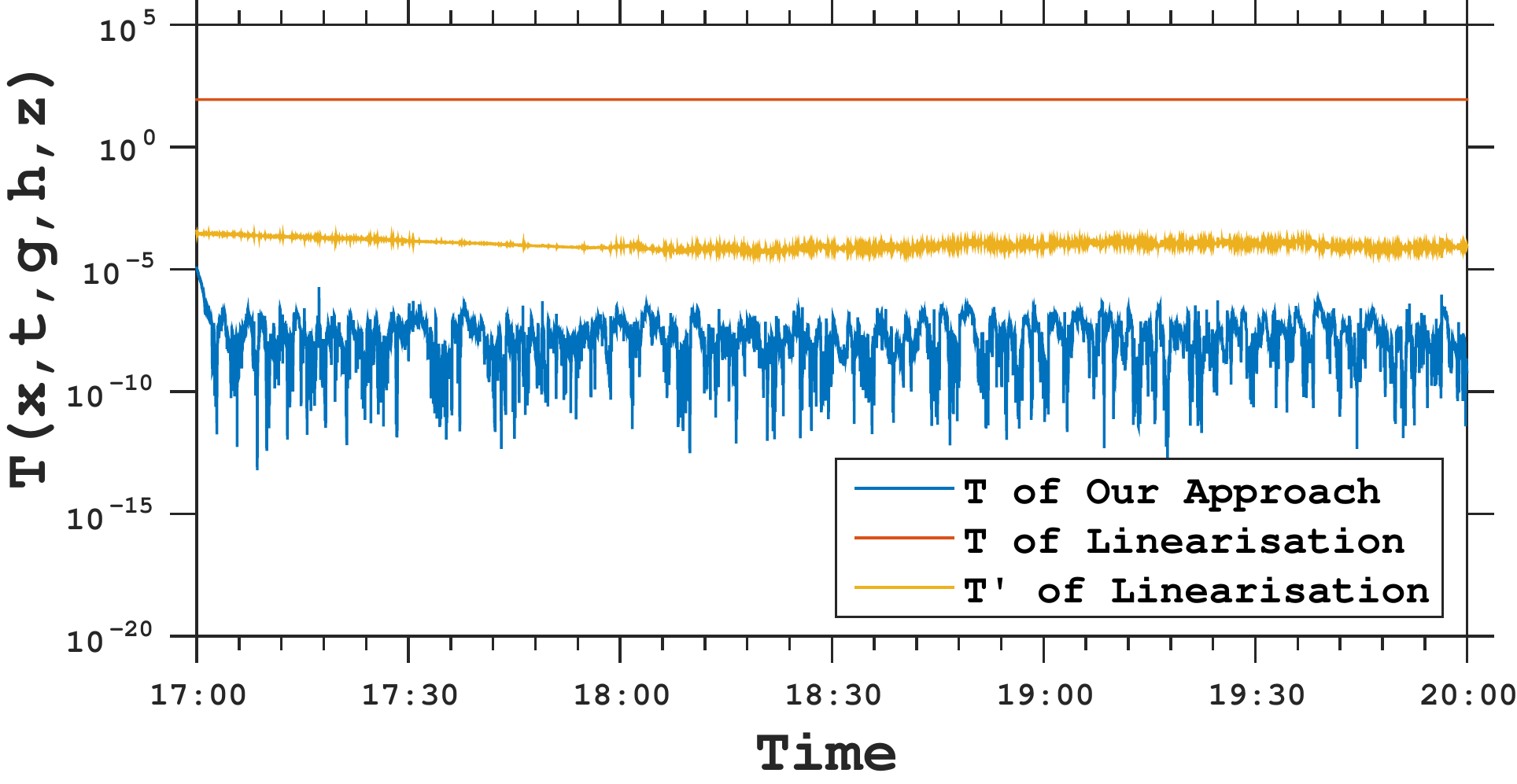}
\caption{A zoom in on the performance on the feeder of Figure~\ref{F_feeder}, from 5pm till 8pm.}
 \label{fig:fivepm}
\end{figure}


\section{A Discussion}

As the volatility of parameters of optimal power flows increases, there is a considerable
interest in the pursuit of solutions to optimal power flows (OPF) in the on-line setting.
In convex optimization and signal processing, related approaches are known as 
warm-starting \cite{Gondzio1998,Yildirim2002,Colombo2011}, %
time-varying convex optimization \cite{7902101}, and dynamic convex optimization \cite{7810574}.
Much of the general-purpose work has, however, focussed on the use of 
interior-point methods \cite{Gondzio1998,Colombo2011,paper3}, 
where a small number of computationally-demanding iterations suffice \cite{Gondzio2012} to reach machine precision. 
Also, no paper we are aware of considered semidefinite programming.
In power systems, much of the work \cite{7480375,7842813,7859385,8013070} has focussed  
 on linearisations of the OPF problem, possibly employing feedback to correct for model mismatches and linearisation errors. 
Although there have been proposals to apply gradient methods \cite{Elia-Allerton13}, Newton method, and L-BFGS \cite{7929408}   
to the general non-convex problem,
as well as proposals to apply gradient algorithms~\cite{7397846} and a related reactive-power control~\cite{7350258}
in the special case of radial networks, 
our approach to the non-convex problem in the on-line setting is novel in a number of ways.

Firstly, we apply coordinate descent to a non-convex Lagrangian, whose solutions under some technical
assumptions coincides with solutions to an SDP relaxation, cf. \cite{Marecek2017}.
Although coordinate-descent algorithms have been used for over half a century, 
the recent interest comes from the improved results \cite{Nesterov2012,Richtarik2014,liu2015asynchronous} on their rates of convergence.
Although the rates of convergence of our algorithm is linear, it is not so-called Nesterov optimal.
For a known Lipschitz constant $L_i$ for each coordinate $i$ and a
 step-size of $1/L_{i^k}$ suggested by Nesterov \cite{Nesterov2012},
one could possibly improve the rate of convergence to:
\begin{align}
\mathbb{E}[ \cL(\xi^{k},\mu) - \cL^*] \leq \left( 1 - \frac{\sigma_\cL }{ d\bar{L}}\right)^k[\cL(\xi^0,\mu) - \cL^*],
\end{align}
where $\bar{L} = \frac{1}{d}\sum_{j=1}^d L_j$. Alternatively, one could pick $i^k$ greedily 
rather than randomly to improve the rate of convergence at the expense of increased per-iteration
 computational effort, 
as discussed in \cite{karimi2016linear}.
In semidefinite programming, low-rank coordinate descent, which considers feasible solutions in the increasing order of rank, 
until one can prove their global optimality, has been first proposed by Burer and Monteiro \cite{Burer2003} and later
analysed by \cite{burer2005,Marecek2017,boumal2016}.
The first application to power systems analysis is due to \cite{Marecek2017}.
As has been shown both here and in \cite{Marecek2017}, the closed-form solution to the coordinate-wise minimisation 
problem allows for excellent computational performance.

Next, for the first time in power-systems literature, we use the Polyak-{\L}ojasiewicz condition
in our analysis.
The condition has been studied since 1960s \cite{Polyak63,lojasiewicz1963propriete},
including a number of variants known as the Kurdyka-{\L}ojasiewicz conditions and 
error bounds \cite{tseng2010approximation}.
The proximal variant we employ was first proposed by Karimi et al.\ \cite{karimi2016linear}.
We imagine that there may be many subsequent applications, 
due to the appeal of allowing for non-convexity and non-unique optima.

In conclusion, coordinate-descent algorithms seem well-suited to tracking solutions of optimal power flows.
In theory, they make it possible to analyse the number of floating-point operations per second a machine
should be capable of, in order to achieve a certain guarantee on the tracking error, while dealing with
a power system of known dimension and loads and limitations of generation of known volatility.
In computational experiments, the algorithm performs very well due to the essentially linear per-iteration
run-time.

\paragraph*{Acknowledgement}
This work has been supported by 
IBM Corporation and  
the National Science Foundation under grants no. NSF:CCF:161871, NSF:CMMI-1663256.

\bibliographystyle{abbrv} 
\bibliography{pscc2018}

\end{document}